\documentclass[reqno,12pt]{amsart}

\usepackage{amssymb}

\newcommand{\alp}{\alpha}
\newcommand{\bet}{\beta}
\newcommand{\gam}{\gamma}
\newcommand{\del}{\delta}
\newcommand{\eps}{\varepsilon}
\newcommand{\kap}{\kappa}

\newcommand{\Gam}{\Gamma}
\newcommand{\Ome}{\Omega}
\newcommand{\Tht}{\Theta}

\newcommand{\F}{\mathbb F}
\newcommand{\N}{\mathbb N}

\DeclareMathOperator{\Prob}{{\mathsf P}}
\DeclareMathOperator{\tr}{Tr}

\newcommand{\longc}{,\ldots,}
\newcommand{\longe}{=\dotsb=}
\newcommand{\longp}{+\dotsb+}

\newcommand{\lpr}{\left(}
\newcommand{\rpr}{\right)}
\newcommand{\lfl}{\left\lfloor}
\newcommand{\rfl}{\right\rfloor}
\newcommand{\lcl}{\left\lceil}
\newcommand{\rcl}{\right\rceil}

\newcommand{\seq}{\subseteq}
\newcommand{\stm}{\setminus}

\newcommand{\simlam}{\stackrel{t}{\sim}}

\theoremstyle{plain}
\newtheorem{theorem}{Theorem}
\newtheorem{corollary}[theorem]{Corollary}
\newtheorem{proposition}[theorem]{Proposition}
\newtheorem{lemma}[theorem]{Lemma}

\newcommand{\refl}[1]{\ref{l:#1}}
\newcommand{\reft}[1]{\ref{t:#1}}
\newcommand{\refc}[1]{\ref{c:#1}}
\newcommand{\refp}[1]{\ref{p:#1}}
\newcommand{\refs}[1]{\ref{s:#1}}
\newcommand{\refb}[1]{\cite{b:#1}}
\newcommand{\refe}[1]{\eqref{e:#1}}

\author[S.~Kopparty]{Swastik Kopparty}
\address{Computer Science and Artificial Intelligence Laboratory, MIT,
  32 Vassar Street, Cambridge, MA 02139, USA}
\email{swastik@mit.edu}

\author[V.~Lev]{Vsevolod F. Lev}
\address{Department of Mathematics, The University of Haifa at Oranim,
  Tivon 36006, Israel}
\email{seva@math.haifa.ac.il}

\author[S.~Saraf]{Shubhangi Saraf}
\address{Computer Science and Artificial Intelligence Laboratory, MIT,
  32 Vassar Street, Cambridge, MA 02139, USA}
\email{shibs@mit.edu}

\author[M.~Sudan]{Madhu Sudan}
\address{Microsoft Research, One Memorial Drive, Cambridge, MA 02142, USA}
\email{madhu@mit.edu}

\title[Kakeya-type sets in finite vector spaces]%
  {Kakeya-type sets \\ in finite vector spaces}

\makeatletter 
\newcommand{\subjclassname@NewMSC}%
  {\textup{2010} Mathematics Subject Classification}
\makeatother  
\subjclass[NewMSC]{Primary: 05B25; secondary: 51E20, 52C17.}


\begin{document}
\baselineskip=16pt

\maketitle

\begin{abstract}
For a finite vector space $V$ and a non-negative integer $r\le\dim V$ we
estimate the smallest possible size of a subset of $V$, containing a
translate of every $r$-dimensional subspace. In particular, we show that if
$K\seq V$ is the smallest subset with this property, $n$ denotes the
dimension of $V$, and $q$ is the size of the underlying field, then for $r$
bounded and $r<n\le rq^{r-1}$ we have $|V\stm K|=\Tht(nq^{n-r+1})$; this
improves previously known bounds $|V\stm K|=\Ome(q^{n-r+1})$ and
 $|V\stm K|=O(n^2q^{n-r+1})$.
\end{abstract}

\section{Introduction and summary of results.}\label{s:intro}

Given a finite vector space $V$ and a non-negative integer $r\le\dim V$, we
say that a subset $K\seq V$ is a \emph{Kakeya set of rank $r$} if it contains
a translate of every $r$-dimensional subspace of $V$; that is, for every
subspace $L\le V$ with $\dim L=r$ there exists a vector $v\in V$ such that
$v+L\seq K$. The goal of this paper is to estimate the smallest possible size
of such a set as a function of the rank $r$, the dimension $\dim V$, and the
size $q$ of the underlying field.

For a prime power $q$, by $\F_q$ we denote the finite field of order $q$.

As shown by Ellenberg, Oberlin, and Tao \cite[Proposition~4.16]{b:eot}, if
$n\ge 2$ is an integer, $q$ a prime power, and $K\seq\F_q^n$ a Kakeya set of
rank $r\in[1,n-1]$, then
  $$ |K| \ge (1-q^{1-r})^{\binom n2} q^n, $$
provided $q$ is sufficiently large as compared to $n$. Our lower bound
presents an improvement of this estimate.
\begin{theorem}\label{t:lower-bound}
If $n\ge r\ge 1$ are integers, $q$ a prime power, and $K\seq\F_q^n$ a Kakeya
set of rank $r$, then
  $$ |K| \ge \left(\frac{q^{r+1}}{q^r+q-1}\right)^n
                                         = \big(1+(q-1)q^{-r}\big)^{-n}q^n. $$
\end{theorem}
The proofs of Theorem~\reft{lower-bound} and most of other results, discussed
in the introduction, are postponed to subsequent sections.

We notice that Theorem~ \reft{lower-bound} extends \cite[Theorem~11]{b:dkss}
and indeed, the latter result is a particular case of the former, obtained
for $r=1$. The proof of Theorem \reft{lower-bound} uses the polynomial method
in the spirit of \cite{b:dkss,b:ss}.

Using the inequality
  $$ (1+x)^{-m} \ge 1-mx;\quad x\ge 0,\ m\ge 1, $$
one readily derives
\begin{corollary}\label{c:lower-bound}
If $n\ge r\ge 1$ are integers, $q$ a prime power, and $K\seq\F_q^n$ a Kakeya
set of rank $r$, then
  $$ |K| \ge \big(1-n(q-1)q^{-r}\big)\,q^n. $$
\end{corollary}

To facilitate comparison between estimates, we introduce the following
terminology. Given two bounds $B_1$ and $B_2$ for the smallest size of a
Kakeya set in $\F_q^n$ (which are either both upper bounds or both lower
bounds), we say that these bounds are \emph{essentially equivalent} in some
range of $n$ and $q$ if there is a constant $C$ such that for all $n$ and $q$
in this range we have
  $$ B_1\le CB_2,\ B_2\le CB_1, $$
and also
  $$ q^n-B_1 \le C(q^n-B_2),\ q^n-B_2\le C(q^n-B_1). $$
We will also say that \emph{the estimates}, corresponding to these
bounds, are  essentially equivalent.

With this convention, it is not difficult to verify that for every fixed
$\eps>0$, the estimates of Theorem \reft{lower-bound} and Corollary
\refc{lower-bound} are essentially equivalent whenever $n\le(1-\eps)q^{r-1}$.
If $n\ge\Big(1+\frac1{q-1}\Big)\,q^{r-1}$, then the estimate of Corollary
\refc{lower-bound} becomes trivial.

Turning to the upper bounds, we present several different constructions. Some 
of them can be regarded as refined and adjusted versions of previously known
ones; other, to our knowledge, did not appear in the literature, but have
been ``in the air'' for a while.

We first present a Kakeya set construction geared towards large fields. It is
based on (i) the ``quadratic residue construction'' due to Mockenhaupt and
Tao \refb{mt} (with a refinement by Dvir, see \refb{ss}), (ii) the ``lifting
technique'' from \refb{eot}, and (iii) the ``tensor power trick''. Our
starting point is \cite[Theorem~8]{b:ss}, stating that if $n\ge 1$ is an
integer and $q$ a prime power, then there exists a rank-$1$ Kakeya set
$K\seq\F_q^n$ such that
\begin{equation}\label{e:Oterm}
  |K| \le 2^{-(n-1)}q^n+O(q^{n-1}),
\end{equation}
with an absolute implicit constant. Indeed, the proof in \refb{ss} yields the
explicit estimate
\begin{equation}\label{e:NOterm}
  |K| \le \begin{cases}
               q\left( \frac{q+1}2 \right)^{n-1} + q^{n-1}
                                                 &\text{ if $q$ is odd}, \\
               (q-1)\left( \frac q2 \right)^{n-1} + q^{n-1}
                                                 &\text{ if $q$ is even}.
          \end{cases}
\end{equation}
This can be used to construct Kakeya sets of rank higher than $1$ using an
observation of Ellenberg, Oberlin, and Tao.
\begin{lemma}[{\cite[Remark~4.19]{b:eot}}]\label{l:EOTlifting}
Let $n\ge r\ge 1$ be integers and $\F$ a field. Suppose that $K_1$ is a
rank-$1$ Kakeya set in the vector space $\F^{n-(r-1)}$, considered as a
subspace of $\F^n$, and let $K:=K_1\cup(\F^n\stm\F^{n-(r-1)})$. Then $K$ is a
Kakeya set of rank $r$ in $\F^n$.
\end{lemma}

Combining \refe{NOterm} with $n=2$ and Lemma \refl{EOTlifting} with $n=r+1$,
we conclude that for every $r\ge 1$ there exists a Kakeya set
$K\seq\F_q^{r+1}$ of rank $r$ such that
\begin{equation}\label{e:rr+1a}
  |K| \le \begin{cases}
               \left( 1 - \frac{q-3}{2q^r} \right) q^{r+1}
                                                 &\text{ if $q$ is odd}, \\
               \left( 1- \frac{q-1}{2q^r} \right) q^{r+1}
                                                 &\text{ if $q$ is even}.
          \end{cases}
\end{equation}
For $q=3$ this estimate is vacuous. However, replacing in this case
\refe{NOterm} with the fact that the vector space $\F_3^2$ contains a
seven-element rank-$1$ Kakeya set, we find a Kakeya set $K\seq\F_3^{r+1}$ of
rank $r$ with
\begin{equation}\label{e:rr+1b}
  |K| \le 3^{r+1} - 2
                 = \left( 1 - \frac{3-(5/3)}{2\cdot3^r} \right) 3^{r+1}.
\end{equation}
Since the product of Kakeya sets of rank $r$ is a Kakeya set of rank $r$ in
the product space, from \refe{rr+1a} and \refe{rr+1b} we derive
\begin{theorem}\label{t:final-upper}
Let $n\ge r\ge 1$ be integers and $q$ a prime power, and write
  $$ \del_q := \begin{cases}
                 3       &\text{ if $q$ is odd and $q\ge 5$}, \\
                 1       &\text{ if $q$ is even}, \\
                 \frac53 &\text{ if $q=3$}.
               \end{cases} $$
There exists a Kakeya set $K\seq\F_q^n$ of rank $r$ such that
  $$ |K| \le \left(1-\frac{q-\del_q}{2q^r} \right)%
                                      ^{\lfloor\frac{n}{r+1}\rfloor} q^n. $$
\end{theorem}

We notice that if $n,r,q$, and $\del_q$ are as in Theorem \reft{final-upper}
and $n>r$, then
  $$ \left(1-\frac{q-\del_q}{2q^r} \right)^{\lfloor\frac{n}{r+1}\rfloor}
       \le 1 - \Ome\big(q^{-(r-1})\big), $$
and that the inequality
  $$ (1-x)^m \le 1-mx+(mx)^2;\quad 0\le x\le 1,\ m\ge 1 $$
shows that if $r<n\le rq^{r-1}$, then indeed
  $$ \left(1-\frac{q-\del_q}{2q^r} \right)^{\lfloor\frac{n}{r+1}\rfloor}
       \le 1 - \Ome\Big(\frac nr\,q^{-(r-1})\Big), $$
with absolute implicit constants. Therefore, we have
\begin{corollary}\label{c:relaxed}
Let $n>r\ge 1$ be integers and $q$ a prime power. There exists a Kakeya set
$K\seq\F_q^n$ of rank $r$ such that
  $$ |K| \le q^n - \Ome\big(q^{n-(r-1)}\big); $$
moreover, if $n\le rq^{r-1}$, then in fact
  $$ |K| \le q^n - \Ome\Big( \frac nr\,q^{n-(r-1)} \Big) $$
(with absolute implicit constants).
\end{corollary}

We remark that Corollaries \refc{lower-bound} and \refc{relaxed} give
nearly matching bounds on the smallest possible size of a Kakeya set of rank
$r$ in $\F_q^n$ in the case where $r$ is fixed, $q$ grows, and the dimension
$n$ does not grow ``too fast''.

The situation where $q$ is bounded and $n$ grows is quite different: for
$r=1$ the $O$-term in \refe{Oterm} do not allow for constructing Kakeya sets
of size $o(q^n)$, and for $r$ large the estimate of
Theorem~\reft{final-upper} is rather weak. Addressing first the case $r=1$,
we develop further the idea behind the proof of \cite[Theorem~8]{b:ss} to
show that the $O$-term just mentioned can be well controlled, making the
result non-trivial in the regime under consideration.
\begin{theorem}\label{t:quadratic}
Let $n\ge 1$ be an integer and $q$ a prime power. There exists a rank-$1$
Kakeya set $K\seq\F_q^n$ with
  $$ |K| < \begin{cases}
             2\big(1+\frac1{q-1}\big)\left( \frac{q+1}2 \right)^n
                                               &\text{ if $q$ is odd}, \\
             \frac32 \big(1+\frac1{q-1}\big)\left( \frac{2q+1}3 \right)^n
                             &\text{ if $q$ is an even power of $2$}, \\
             \frac32 \left( \frac{2(q+\sqrt q+1)}{3} \right)^n
                                   &\text{ if $q$ is an odd power of $2$}.
           \end{cases} $$
\end{theorem}

Theorem~\reft{quadratic} is to be compared against the case $r=1$ of
Theorem~\reft{lower-bound} showing that if $K\seq\F_q^n$ is a rank-$1$ Kakeya
set, then $|K|\ge\big(q^2/(2q-1)\big)^n$.

For several small values of $q$ the estimate of Theorem \reft{quadratic} can
be improved using a combination of the ``missing digit construction'' and the
``random rotation trick'' of which we learned from Terry Tao who, in turn,
refers to Imre Ruzsa (personal communication in both cases).

For a field $\F$, by $\F^\times$ we denote the set of non-zero elements of
$\F$.

The missing digit construction by itself gives a very clean, but rather weak
estimate.
\begin{theorem}\label{t:missing-digit}
Let $n\ge 1$ be an integer and $q$ a prime power, and suppose that
$\{e_1\longc e_n\}$ is a linear basis of $\F_q^n$. Let
\begin{align*}
  A &:= \{ \eps_1e_1\longp\eps_n e_n
                         \colon \eps_1\longc\eps_n \in \F_q^\times \}
\intertext{and}
  B &:= \{ \eps_1e_1\longp\eps_n e_n
                         \colon \eps_1\longc\eps_n \in \{0,1\} \}.
\end{align*}
Then $K:=A\cup B$ is a rank-$1$ Kakeya set in $\F_q^n$ with
  $$ |K| = (q-1)^n+2^n-1. $$
\end{theorem}

Using the random rotation trick, we boost Theorem \reft{missing-digit} to
\begin{theorem}\label{t:random-rotations}
Let $n\ge 1$ be an integer and $q\ge 3$ a prime power. There exists a
rank-$1$ Kakeya set $K\seq\F_q^n$ such that
  $$ |K| < \Big( \frac q{2^{2/q}} \Big)^{n+O\big(\sqrt{n\ln q/q}\big)} $$
(with an absolute implicit constant).
\end{theorem}

To compare Theorems \reft{quadratic} and \reft{random-rotations} we notice
that $(q+1)/2<2^{-2/q}q$ for every integer $q\ge 4$, that
$(2q+1)/3<2^{-2/q}q$ for every integer $q\ge 5$, and that
 $2(q+\sqrt q+1)/3<2^{-2/q}q$ for every integer $q\ge 14$. Thus, for $q$
fixed and $n$ growing, Theorem \reft{quadratic} supersedes
Theorem~\reft{random-rotations} except if $q\in\{3,4,8\}$. Indeed, the remark
following the proof of Proposition~\refp{qodd}
(Section~\refs{proofs-quadratic}) shows that the value $q=8$ can be removed
from this list.

Finally, we return to constructions of Kakeya sets of rank $r\ge 2$. As
remarked above, for $r$ large the bound of Theorem~\reft{final-upper} (and
consequently, that of Corollary~\refc{relaxed}) is rather weak. The best
possible construction we can give in this regime does not take linearity into
account and is just a \emph{universal set} construction where, following
\refb{abs}, we say that a subset of a group is $k$-universal if it contains a
translate of every $k$-element subset of the group. As shown in \refb{abs},
every finite abelian group $G$ possesses a $k$-universal subset of size at
most $8^{k-1}k|G|^{1-1/k}$. In our present context the group under
consideration is the additive group of the vector space $\F_q^n$, in which
case we were able to give a particularly simple construction of universal
sets and refine slightly the bound just mentioned.
\begin{lemma}\label{l:universal}
Let $q$ be a prime power and $n,k\ge 1$ integers satisfying $k\le q^n$. There
exists a set $U\seq\F_q^n$ with
  $$ |U| = \big( 1 - \big( 1-q^{-\lfl n/k\rfl} \big)^k \big) q^n $$
such that $U$ contains a translate of every $k$-element subset of
$\F_q^n$.
\end{lemma}

Aa an immediate consequence we have
\begin{theorem}\label{t:universal}
Let $n\ge r\ge 1$ be integers and $q$ a prime power. There exists a Kakeya
set $K\seq\F_q^n$ of rank $r$ such that
  $$ |K| \le \big( 1 - \big( 1-q^{-\lfl n/q^r\rfl} \big)^{q^r} \big) q^n. $$
\end{theorem}

Using the estimates $\lfl n/q^r\rfl> n/q^r-1$ and $(1-x)^m\ge 1-mx$ (applied
with $x=q^{-\lfl n/q^r\rfl}$ and $m=q^r$), we obtain
\begin{corollary}\label{c:universal}
Let $n\ge r\ge 1$ be integers and $q$ a prime power. There exists a Kakeya
set $K\seq\F_q^n$ of rank $r$ such that
  $$ |K| < q^{n(1-q^{-r})+r+1}. $$
\end{corollary}

It is not difficult to verify that Corollary~\refc{universal} supersedes
Corollary~\refc{relaxed} for $n\ge(r+2)q^r$, and that for $n$ growing,
Theorem \reft{universal} supersedes Theorem \reft{final-upper} if $r$ is
sufficiently large as compared to $q$ (roughly, $r>Cq/\log q$ with a suitable
constant $C$).

A slightly more precise version of Corollary~\refc{universal} is that there
exists a Kakeya set $K\seq\F_q^n$ of rank $r$ with
  $$ |K| \le q^{n-\lfl n/q^r\rfl+r}; $$
this is essentially equivalent to Theorem \reft{universal} provided that
$n\ge (r+1)q^r$. (On the other hand, Theorem \reft{universal} becomes trivial
if $n<q^r$.)

The remainder of the paper is mostly devoted to the proofs of
Theorems~\reft{lower-bound}, \reft{quadratic}, \reft{missing-digit}, and
\reft{random-rotations}, and Lemma~\refl{universal}. For the convenience of
the reader and self-completeness, we also prove (a slightly generalized
version of) Lemma~\refl{EOTlifting} in the Appendix. Section
\refs{conclusion} contains a short summary and concluding remarks.

\section{Proof of Theorem \reft{lower-bound}.}\label{s:proofs}

As a preparation for the proof of Theorem \reft{lower-bound}, we briefly
review some basic notions and results related to the polynomial method; the
reader is referred to \refb{dkss} for an in-depth treatment and proofs.

For the rest of this section we use multidimensional formal variables,
which are to be understood just as $n$-tuples of ``regular'' formal
variables with a suitable $n$. Thus, for instance, if $n$ is a positive
integer and $\F$ is a field, we can write $X=(X_1\longc X_n)$ and
$P\in\F[X]$, meaning that $P$ is a polynomial in the $n$ variables
$X_1\longc X_n$ over $\F$. By $\N_0$ we denote the set of non-negative
integers, and for $X$ as above and an $n$-tuple $i=(i_1\longc
i_n)\in\N_0^n$ we let $\|i\|:=i_1\longp i_n$ and $X^i:=X_1^{i_1}\dotsb
X_n^{i_n}$.

Let $\F$ be a field, $n\ge 1$ an integer, and $X=(X_1\longc X_n)$ and
$Y=(Y_1\longc Y_n)$ formal variables. To every polynomial $P$ in $n$
variables over $\F$ and every $n$-tuple $i\in\N_0^n$ there corresponds a
uniquely defined polynomial $P^{(i)}$ over $\F$ in $n$ variables such
that
  $$ P(X+Y) = \sum_{i\in\N_0^n} P^{(i)}(Y) X^i. $$
The polynomial $P^{(i)}$ is called \emph{the Hasse derivative of $P$ of order
$i$}. Notice, that $P^{(0)}=P$ (which follows, for instance, by letting
$X=(0\longc 0)$), and if $\|i\|>\deg P$, then $P^{(i)}=0$. Also, it is easy
to check that if $P_H$ denotes the homogeneous part of $P$ (meaning that
$P_H$ is a homogeneous polynomial such that $\deg(P-P_H)<\deg P$), and
$(P^{(i)})_H$ denotes the homogeneous part of $P^{(i)}$, then
$(P^{(i)})_H=(P_H)^{(i)}$.

A polynomial $P$ in $n$ variables over a field $\F$ is said to vanish at
a point $a\in\F^n$ with multiplicity $m$ if $P^{(i)}(a)=0$ for each
$i\in\N_0^n$ with $\|i\|<m$, whereas there exists $i\in\N_0^n$ with
$\|i\|=m$ such that $P^{(i)}(a)\ne0$. In this case $a$ is also said to be
a zero of $P$ of multiplicity $m$. We denote the multiplicity of zero of
a non-zero polynomial $P$ at $a$ by $\mu(P,a)$; thus, $\mu(P,a)$ is the
largest integer $m$ with the property that
  $$ P(X+a) = \sum_{i\in\N_0^n\colon \|i\|\ge m} c(i,a)X^i;
                                                       \quad c(i,a)\in \F. $$

\begin{lemma}[{\cite[Lemma~5]{b:dkss}}]\label{l:multder}
Let $n\ge 1$ be an integer. If $P$ is a non-zero polynomial in $n$
variables over the field $\F$ and $a\in\F^n$, then for any $i\in\N_0^n$
we have
  $$ \mu(P^{(i)},a) \ge \mu(P,a)-\|i\|. $$
\end{lemma}

\begin{lemma}[{\cite[Proposition 10]{b:dkss}}]\label{l:dimmult}
Let $n,m\ge 1$ and $k\ge 0$ be integers, and $\F$ a field. If a finite set
$S\seq\F^n$ satisfies $\binom{m+n-1}{n}\,|S|<\binom{n+k}{n}$, then there is a
non-zero polynomial over $\F$ in $n$ variables of degree at most $k$,
vanishing at every point of $S$ with multiplicity at least $m$.
\end{lemma}

Yet another lemma we need is a direct corollary of
\cite[Proposition~6]{b:dkss}.
\begin{lemma}\label{l:composition}
Let $n,r\ge 1$ be integers and $P$ a non-zero polynomial in $n$ variables
over the field $\F$, and suppose that $b,d_1\longc d_r\in\F^n$. Then for any
$t_1\longc t_r\in\F$ we have
  $$ \mu(P(b+T_1d_1\longp T_rd_r), (t_1\longc t_r))
                                        \ge \mu(P, b+t_1d_1\longp t_rd_r), $$
where $P(b+T_1d_1\longp T_rd_r)$ is a polynomial in the formal variables
$T_1\longc T_r$.
\end{lemma}

The multiplicity Schwartz-Zippel lemma is as follows.
\begin{lemma}[{\cite[Lemma 8]{b:dkss}}]
Let $n\ge 1$ be an integer, $P$ a non-zero polynomial in $n$ variables over a
field $\F$, and $S\seq\F$ a finite set. Then
  $$ \sum_{z\in S^n} \mu(P,z) \le \deg P \cdot |S|^{n-1}. $$
\end{lemma}

\begin{corollary}\label{c:sbe}
Let $n\ge 1$ be an integer, $P$ a non-zero polynomial in $n$ variables over a
field $\F$, and $S\seq\F$ a finite set. If $P$ vanishes at every point of
$S^n$ with multiplicity at least $m$, then
 $\deg P\ge m|S|$.
\end{corollary}

We are now ready to prove Theorem~\reft{lower-bound}.

\begin{proof}[Proof of Theorem \reft{lower-bound}]
Assuming that $m$ and $k$ are positive integers with
\begin{equation}\label{e:loc24}
  k < q^r \lcl \frac{qm-k}{q-1} \rcl
\end{equation}
(no typo: $k$ enters both sides!), we show first that
\begin{equation}\label{e:loc25}
  \binom{m+n-1}{n}\,|K|\ge\binom{n+k}{n},
\end{equation}
and then optimize by $m$ and $k$.

Suppose for a contradiction that \refe{loc25} fails; thus, by Lemma
\refl{dimmult}, there exists a non-zero polynomial $P$ over $\F_q$ of
degree at most $k$ in $n$ variables, vanishing at every point of $K$ with
multiplicity at least $m$.

Write $l:=\lcl \frac{qm-k}{q-1} \rcl$ and fix $i=(i_1\longc
i_n)\in\N_0^n$ satisfying $w:=\|i\|<l$. Let $Q:=P^{(i)}$, the $i$th Hasse
derivative of $P$.

Since $K$ is a Kakeya set of rank $r$, for every $d_1\longc d_r\in\F_q^n$
there exists $b\in\F_q^n$ such that $b+t_1d_1\longp t_rd_r\in K$ for all
$t_1\longc t_k\in\F_q$; hence,
  $$ \mu(P,b+t_1d_1\longp t_rd_r)\ge m,  $$
and therefore, by Lemma \refl{multder},
  $$ \mu(Q,b+t_1d_1\longp t_rd_r)\ge m-w $$
whenever $t_1\longc t_r\in\F_q$. By Lemma \refl{composition}, we have
  $$ \mu(Q,b+t_1d_1\longp t_rd_r) \le
              \mu(Q(b+T_1d_1\longp T_rd_r),(t_1\longc t_r)), $$
where $Q(b+T_1d_1\longp T_rd_r)$ is considered as a polynomial in the
variables $T_1\longc T_r$. Thus, for every $d_1\longc d_r\in\F_q^n$ there
exists $b\in\F_q^n$ such that $Q(b+T_1d_1\longp T_rd_r)$ vanishes with
multiplicity at least $m-w$ at each point $(t_1\longc t_r)\in\F_q^r$.
Compared with
  $$ \deg Q(b+T_1d_1\longp T_rd_r) \le \deg Q\le k-w < q (m-w) $$
(as it follows from $w<l$), in view of Corollary~\refc{sbe} this shows that
$Q(b+T_1d_1\longp T_rd_r)$ is the zero polynomial.

Let $P_H$ and $Q_H$ denote the homogeneous parts of the polynomials $P$ and
$Q$, respectively, so that $Q(b+T_1d_1\longp T_rd_r)=0$ implies
$Q_H(T_1d_1\longp T_rd_r)=0$. Thus, $(P_H)^{(i)}(T_1d_1\longp T_rd_r)=0$ for
all $d_1\longc d_r\in\F_q^n$. We interpret this saying that $(P_H)^{(i)}$,
considered as a polynomial in $n$ variables over the field of rational
functions $\F_q(T_1\longc T_r)$, vanishes at every point of the set
  $$ \{T_1d_1\longp T_rd_r\colon d_1\longc d_r\in\F_q^n\} = S^n, $$
where
  $$ S := \{\alp_1T_1\longp\alp_r T_r\colon \alp_1\longc\alp_r\in\F_q\}. $$

This shows that all Hasse derivatives of $P_H$ of order, smaller than $l$,
vanish on $S^n$; in other words, $P_H$ vanishes with multiplicity at least
$l$ at every point of $S^n$. Since, on the other hand, by \refe{loc24} we
have
  $$ \deg P_H = \deg P \le k < q^rl = |S|l, $$
from Corollary \refc{sbe} we conclude that $P_H$ is the zero polynomial,
which is wrong as the homogeneous part of a non-zero polynomial is non-zero.

Thus, \refe{loc25} is established. Rewriting it as
  $$ |K| \ge \frac{(k+1)(k+2)\ldots(k+n)}{m(m+1)\ldots(m+n-1)}, $$
to optimize we choose $k=Nq^{r+1}-1$ and $m=(q^r+q-1)N$, where $N$ is a
positive integer. With this choice, inequality \refe{loc24} is satisfied for
any values of $N$, and the assertion of Theorem \reft{lower-bound} follows
from the observation that the limit of the right-hand side as $N\to\infty$ is
$(q^{r+1}/(q^r+q-1))^n$.
\end{proof}

\section{Proof of Theorem \reft{quadratic}.}\label{s:proofs-quadratic}

For a field $\F$, a function $f\colon\F\to\F$, and an element $ t\in\F$, we
write
  $$ I_f( t) := \{ f(x)+ t x\colon x\in\F \}. $$

Our proof of Theorem \reft{quadratic} relies on the following lemma, a
provisional form of which is implicitly contained in \refb{ss}.
\begin{lemma}\label{l:If}
Let $n\ge 1$ be an integer, $\F$ a finite field, and $f\colon\F\to\F$ a
non-linear function. There exists a rank-$1$ Kakeya set $K\seq\F^n$ with
  $$ |K| = \sum_{ t\in\F} \frac{|I_f( t)|^n-1}{|I_f( t)|-1}. $$
\end{lemma}

\begin{proof}
Let
  $$ K := \{ (x_1\longc x_j, t,0\longc 0)\colon 0\le j\le n-1,
                  \  t\in\F,\ x_1\longc x_j\in I_f( t) \}. $$
Since $f$ is non-linear, we have $|I_f( t)|>1$ for each $ t\in\F$, and it
follows that
  $$ |K| = \sum_{j=0}^{n-1} \sum_{ t\in\F} |I_f( t)|^j
               =  \sum_{ t\in\F} \frac{|I_f( t)|^n-1}{|I_f( t)|-1}. $$

To show that $K$ is a rank-$1$ Kakeya set we prove that it contains a line in
every direction $d=(d_1\longc d_n)\in\F^n\stm\{0\}$. Without loss of
generality we assume that, for some $j\in[1,n-1]$, we have $d_{j+1}=1$ and
$d_{j+2}\longe d_n=0$, and we let
  $$ b := (f(d_1)\longc f(d_j),0\longc 0). $$
For every $ t\in\F$ we have then
  $$ b + td = (f(d_1)+ td_1\longc f(d_j)+td_j, t,0\longc 0) \in K, $$
completing the proof.
\end{proof}

The assertion of Theorem~\reft{quadratic} for $q$ odd follows immediately
from Lemma \refl{If} upon choosing $\F:=\F_q$ and $f(x):=x^2$, and observing
that then $|I_f( t)|=(q+1)/2$ for each $ t\in\F$ in view of
  $$ x^2+ t x = (x+ t/2)^2 -  t^2/4. $$
In the case of $q$ even the assertion follows easily by combining
Lemma~\refl{If} with the following two propositions.

\begin{proposition}\label{p:qeven}
Suppose that $q$ is an even power of $2$ and let $f(x):=x^3\ (x\in\F_q)$.
Then for every $ t\in\F_q$ we have $|I_f( t)|\le(2q+1)/3$.
\end{proposition}

\begin{proposition}\label{p:qodd}
Suppose that $q$ is an odd power of $2$ and let
 $f(x):=x^{q-2}+x^2\ (x\in\F_q)$. Then for every $t\in\F_q$ we have
$|I_f(t)|\le2(q+\sqrt q+1)/3$.
\end{proposition}

To complete the proof of Theorem~\reft{quadratic} it remains to prove
Propositions~\refp{qeven} and~\refp{qodd}. For this we need the following
well-known fact.
\begin{lemma}\label{l:qeq}
Suppose that $q$ is a power of $2$, and let $\tr$ denote the trace function
from the field $\F_q$ to its two-element subfield. For
$\alp,\bet,\gam\in\F_q$ with $\alp\ne 0$, the number of solutions of the
equation $\alp x^2+\bet x+\gam=0$ in the variable $x\in\F_q$ is
  $$ \begin{cases}
       1 &\text{ if $\bet=0$}, \\
       0 &\text{ if $\bet\ne 0$ and $\tr(\alp\gam/\bet^2)=1$}, \\
       2 &\text{ if $\bet\ne 0$ and $\tr(\alp\gam/\bet^2)=0$}.
     \end{cases} $$
\end{lemma}

\begin{proof}[Proof of Proposition~\refp{qeven}]
The assumption that $q$ is an even power of $2$ implies  that $q-1$ is
divisible by $3$. Consequently, $\F_q$ contains $(q-1)/3+1<(2q+1)/3$ cubes,
and we assume below that $t\ne 0$.

For $x,y\in\F_q$ we write $x\sim y$ if $x^3+ t x=y^3+ t y$. Clearly, this
defines an equivalence relation on $\F_q$, and $|I_f( t)|$ is just the number
of equivalence classes. Since the equation $x^3+ t x=0$ has exactly two
solutions, which are $0$ and $\sqrt t$, the set $\{0,\sqrt t\}$ is an
equivalence class. Fix now $x\notin\{0,\sqrt t\}$ and consider the
equivalence class of $x$. For $x\sim y$ to hold it is necessary and
sufficient that either $y^2+xy+x^2= t$, or $x=y$, and these two conditions
cannot hold simultaneously in view of $x\ne\sqrt t$. Hence, with $\tr$
defined as in Lemma~\refl{qeq}, and using the assertion of the lemma, the
number of elements in the equivalence class of $x$ is
  $$ \begin{cases}
       1 &\text{if $\tr((x^2+ t)/x^2)=1$}, \\
       3 &\text{if $\tr((x^2+ t)/x^2)=0$}.
     \end{cases} $$
As $x$ runs over all elements of $\F_q\stm\{0,\sqrt t\}$, the expression
$(x^2+ t)/x^2$ runs over all elements of $\F_q\stm\{0,1\}$. Since $q$ is an
even power of $2$, we have $\tr(1)=\tr(0)=0$; thus, there are $q/2-2$ values
of $x\notin\{0,\sqrt t\}$ with $\tr((x^2+ t)/x^2)=0$.

To summarize, $q/2-2$ elements of $\F_q$ are contained in three-element
equivalence classes, the elements $0$ and $\sqrt t$ form a two-element class,
and the remaining $q/2$ elements lie in one-element classes. It follows that
the number of classes is
  $$ \frac{q/2-2}3 + 1 + q/2 = \frac{2q+1}3. $$
\end{proof}

\begin{proof}[Proof of Proposition~\refp{qodd}]
We define the equivalence relation $\sim$ and the trace function $\tr$ on
$\F_q$ as in the proof of Proposition~\refp{qeven}. Notice, that the
assumption that $q$ is an odd power of $2$ implies that $q-1$ is not
divisible by $3$, whence the cube function $x\mapsto x^3$ is a bijection of
$\F_q$ onto itself. Furthermore, we have $x^{q-2}=x^{-1}$ for
$x\in\F_q^\times$, implying
  $$ I_f( t) = \{x^{-1}+x^2+ t x\colon x\in\F_q^\times \} \cup \{0\}. $$

Suppose first that $ t=0$, in which case
  $$ I_f(0)=\{x^{-1}+x^2\colon x\in\F_q^\times\} $$
in view of $1^{-1}+1^2=0$. As simple computation shows that $x\sim y$ with
$x,y\in\F_q^\times,\ x\ne y$ holds if and only if $1/(xy)=x+y$; that is,
$xy^2+x^2y+1=0$. For $x\in\F_q^\times$ fixed, this equation in $y$ has, by
Lemma~\refl{qeq}, two (non-zero) solutions is $\tr(1/x^3)=0$, and no
solutions if $\tr(1/x^3)=1$. It follows that each $x\in\F_q^\times$ contains
either three, or one non-zero element in its equivalence class, according to
whether $\tr(1/x^3)=0$ or $\tr(1/x^3)=1$. By a remark at the beginning of the
proof, as $x$ runs over all elements of $\F_q^\times$, so does $1/x^3$.
Hence, there are exactly $q/2-1$ those $x\in\F_q^\times$ with $\tr(1/x^3)=0$,
and $q/2$ those $x\in\F_q^\times$ with $\tr(1/x^3)=1$. Consequently,
$|I_f(0)|$, which is the number of equivalence classes, is equal to
  $$ \frac{q/2-1}3 + q/2 = \frac{2q-1}3. $$

For the rest of the proof we assume that $ t\ne 0$.

The equation $x^{-1}+x^2+ t x= t^{-1}$ is easily seen to have the solution
set $\{ t,1/\sqrt t\}$ which, therefore, is an equivalence class, consisting
of two elements if $ t\ne 1$ or just one element if $ t=1$. Fix
$x\in\F_q^\times\stm\{ t,1/\sqrt t\}$. For $y\in\F_q^\times,\ y\neq x$, we
have $x\sim y$ if and only if $1/(xy)=x+y+ t$; equivalently, $xy^2+x(x+
t)y+1=0$. This equation has two solutions (distinct from $x$ and $0$) if
$\tr(1/x(x+ t)^2)=0$, and no solutions if $\tr(1/x(x+ t)^2)=1$. In the former
case the equivalence class of $x$ contains three non-zero elements, and,
consequently, if we let
  $$ N := \# \big\{ x\in \F_q^\times\stm\{ t,1/\sqrt t\}
                                 \colon \tr\big(1/(x(x+ t)^2)\big)=0 \big\}, $$
then
\begin{equation}\label{e:Iflam}
  |I_f( t)| \le
        \begin{cases}
          q - \frac23\,N     &\text{ if $ t=1$}, \\
          q - \frac23\,N - 1 &\text{ if $ t\ne 1$}.
        \end{cases}
\end{equation}
To estimate $N$ we notice that
  $$ \frac1{x(x+ t)^2} = \frac1{ t^2x} + \frac1{ t^2(x+ t)}
                                                 + \frac1{ t(x+ t)^2}, $$
and that
  $$ \tr\left(\frac1{ t(x+ t)^2}\right)
                               = \tr\left(\frac1{\sqrt t(x+ t)}\right), $$
implying
\begin{align*}
  \tr\left(\frac1{x(x+ t)^2}\right)
       &= \tr\left(\frac1{ t^2x}
           + \left(\frac1{ t^2}+\frac1{\sqrt t}\right)
                                                    \frac1{x+ t} \right) \\
       &= \tr \left( \frac{x/\sqrt t+1/ t}{x(x+ t)} \right).
\end{align*}
Thus, if $ t=1$, then
  $$ \tr\left(\frac1{x(x+ t)^2}\right) = \tr\left(\frac1x\right), $$
showing that
  $$ N = \# \{x\in\F_q\stm\{0,1\}\colon \tr(1/x)=0 \} = q/2-1 $$
(as the assumption that $q$ is an odd power of $2$ implies $\tr(1)=1$), and
hence
  $$ |I_f(1)| \le q - \frac23\,(q/2-1) = \frac{2q+2}3 $$
by \refe{Iflam}.

Finally, suppose that $ t\notin\{0,1\}$. For brevity we write
  $$ R(x) := \frac{x/\sqrt t+1/ t}{x(x+ t)}, $$
and let $\psi$ denote the additive character of the field $\F_q$, defined by
  $$ \psi(x) = (-1)^{\tr(x)};\ x\in\F_q. $$
Since $R(1/\sqrt t)=0$, we have
\begin{align*}
  N &= \frac12 \sum_{x\in\F_q\stm\{0, t,1/\sqrt t\}}
                                                \big(1+\psi(R(x))\big) \\
    &= \frac12 \sum_{x\in\F_q\stm\{0, t\}} \psi(R(x)) + \frac q2 - 2.
\end{align*}
Using Weil's bound (as laid out, for instance, in \cite[Theorem~2]{b:mm}), we
get
  $$ N \ge \frac q2-2 - \frac12\,(2\sqrt q+1)
                                          = \frac q2 - \sqrt q - \frac 52. $$
Now \refe{Iflam} gives
  $$ |I_f( t)| \le q - \frac23\,\big((q/2)-\sqrt q-(5/2)\big) - 1
              = \frac{2(q+\sqrt q+1)}3, $$
which completes the proof.
\end{proof}

We remark that for any particular prime power $q$ the estimates of
Propositions \refp{qeven} and \refp{qodd} can (potentially) be improved by
computing the exact values of the quantities $|I_f(t)|$. Say, a direct
inspection shows that for $q=8$ and $f(x):=x^6+x^2$ one has $|I_f(t)|\le 6$
for each $t\in\F_8$; consequently, for every integer $n\ge 1$ the vector
space $\F_8^n$ possesses a rank-$1$ Kakeya set of size smaller than
$\frac85\cdot6^n$.

A natural question arising in connection with our proof of
Theorem~\reft{quadratic} is whether and to which extent the result can be
improved by choosing ``better'' functions $f$ in Propositions~\refp{qeven}
and~\refp{qodd} and in the application of Lemma \refl{qeq} in the case of $q$
odd. We conclude this section showing that we have almost reached the limits
of the method.

\begin{lemma}
For every prime power $q$ and function $f\colon\F_q\to\F_q$, there exists an
element $t\in\F_q$ with
  $$ |I_f( t)| > q/2. $$
\end{lemma}

\begin{proof}
For $x,y, t\in\F_q$ we write $x\simlam y$ if $f(x)+ t x=f(y)+ t y$;
equivalently, if either $x=y$, or $x\ne y$ and $(f(x)-f(y))/(x-y)=- t$. It
follows from the first form of this definition that $\simlam$ is an
equivalence relation on $\F_q$ and $|I_f( t)|$ is the number of equivalence
classes, and from the second form that for every pair $(x,y)$ with $x\ne y$
there exists a unique $ t\in\F_q$ with $x\simlam y$.

For each $ t\in\F_q$, consider the graph $\Gam_ t$ on the vertex set $\F_q$,
in which two vertices $x\ne y$ are adjacent if and only if
 $x\simlam y$. By the remark just made, every edge of the complete graph on
the vertex set $\F_q$ belongs to exactly one graph $\Gam_ t$. Consequently,
there exists $ t\in\F_q$ such that the number of edges of $\Gam_ t$, which we
denote by $e(\Gam_ t)$, does not exceed $q^{-1}\binom q2=(q-1)/2$. By the
construction, the graph $\Gam_ t$ is a disjoint union of cliques; let $k$
denote the number, and $m_1\longc m_k$ the sizes of these cliques. Thus, we
have
  $$ m_1\longp m_k=q\quad\text{and}\quad |I_f( t)|=k, $$
and it remains to show that $k>q/2$. We distinguish two cases.

If $q$ is even then, using convexity, we get
  $$ \frac q2-1 \ge e(\Gam_ t) = \binom{m_1}2\longp\binom{m_k}2
        \ge k \binom{q/k}2 = \frac12\,q\lpr\frac qk-1\rpr, $$
whence
  $$ q-1 > \frac{q^2}{2k}, $$
leading to the desired bound.

If $q$ is odd, we let
  $$ s := \# \{ i\in[1,k]\colon m_i=1 \}
        \quad\text{and}\quad l := \# \{ i\in[1,k]\colon m_i\ge 2 \}, $$
so that $s+l=k$ and
\begin{equation}\label{e:hren28}
  s + 2l \le q.
\end{equation}
Then
\begin{multline*}
  \frac{q-1}2 \ge e(\Gam_ t)
    = \sum_{i\in[1,k]\colon m_i\ge 2}\binom{m_i}2 \ge l \binom{(q-s)/l}2 \\
        = \frac12\,(q-s)\lpr \frac{q-s}l - 1 \rpr = \frac1{2l}\,(q-s)(q-k).
\end{multline*}
If we had $k\le q/2$, this would yield
  $$ \frac q2 > \frac{q-1}2 \ge \frac1{2l}\,(q-s)\cdot\frac q2, $$
contradicting \refe{hren28}.
\end{proof}

\section{Proof of Theorems \reft{missing-digit} and \reft{random-rotations}.}
  \label{s:proofs-mdrr}

\begin{proof}[Proof of Theorem \reft{missing-digit}]
Given a vector $d=\eps_1e_1\longp\eps_ne_n$ with $\eps_1\longc\eps_n\in\F_q$,
let
  $$ b := \sum_{i\in[1,n]\colon \eps_i=0} e_i. $$
Thus, $b\in B$, and it is readily verified that for $t\in\F_q^\times$ we have
$b+td\in A$. Therefore, the line through $b$ in the direction $d$ is entirely
contained in $K$.

The assertion on the size of $K$ follows from $A\cap B=\{e_1\longp e_n\}$.
\end{proof}

\begin{proof}[Proof of Theorem \reft{random-rotations}]
We notice that the assertion is trivial if $n=O(q(\ln q)^3)$, as in this case
for a sufficiently large constant $C$ we have
  $$ \Big(\frac q{2^{2/q}}\Big)^{n+C\sqrt{n\ln q/q}} > q^n; $$
consequently, we assume
\begin{equation}\label{e:n-large1}
  n > 32 q(\ln q)^3
\end{equation}
for the rest of the proof.

Fix a linear basis $\{e_1\longc e_n\}\seq\F_q^n$ and, as in
Theorem~\reft{missing-digit}, let
\begin{align*}
  A &:= \{ \eps_1e_1\longp\eps_n e_n
                                \colon \eps_1\longc\eps_n \in \F_q^\times \}
\intertext{and}
  B &:= \{ \eps_1e_1\longp\eps_n e_n
                         \colon \eps_1\longc\eps_n \in \{0,1\} \}.
\end{align*}
Given a vector $v=\eps_1e_1\longp\eps_ne_n$ with $\eps_1\longc\eps_n\in\F_q$
and a scalar $\eps\in\F_q$, let $\nu_\eps(v)$ denote the number of those
indices $i\in[1,n]$ with $\eps_i=\eps$. Set $\del:=2\sqrt{\ln q}$ and define
  $$ D_0 := \{ d \in \F_q^n \colon \nu_\eps(d) > n/q - \del(n/q)^{1/2}
   \ \text{for all}\ \eps\in\F_q\} $$
and
  $$ A_0 := \{ a\in A\colon \nu_1(a) > 2n/q - 2\del(n/q)^{1/2} \}. $$

Suppose that a vector $v\in\F_q^r$ is chosen at random, with equal
probability for each vector to be chosen. For each fixed $\eps\in\F_q$, the
quantity $\nu_\eps(v)$ is then a random variable, distributed binomially with
the parameters $n$ and $1/q$. As a result, using standard estimates for the
binomial tail (as, for instance, \cite[Theorem~A.1.13]{b:as}), we get
  $$ \Prob\big(\nu_\eps(v)\le n/q-\del(nq)^{1/2})\big)
                              \le e^{-\del^2(n/q)/(2n/q)} = \frac1{q^2}. $$
Consequently, the probability of a vector, randomly drawn from $\F_q^n$, not
to belong to $D_0$, is at most $1/q$, for which reason we call the elements
of $D_0$ \emph{popular directions}.

If $d=\eps_1e_1\longp\eps_ne_n\in D_0$ then, letting
$b:=\sum_{i\in[1,n]\colon\eps_i=0} e_i$, for each $t\in\F_q^\times$ we have
  $$ \nu_1(b+td) = \nu_0(d) + \nu_{t^{-1}}(d) > 2n/q-2\del(n/q)^{1/2}, $$
whence $b+td\in A_0$. Thus, the set $K_0:=B\cup A_0$ contains a line in every
popular direction.

To estimate the size of $K_0$ we notice that, letting
 $N:=\lfl 2n/q-2\del(n/q)^{1/2}\rfl+1$, we have
  $$ |A_0| = \sum_{j=N}^n \binom{n}j (q-2)^{n-j}. $$
Assumption \refe{n-large1} implies that the summands in the right-hand
side decay as $j$ grows, whence
  $$ |A_0| \le n \binom nN (q-2)^{n-N}. $$
Consequently, writing
 $$ H(x) := x\ln(1/x)+(1-x)\ln(1/(1-x)),\quad x\in(0,1) $$
and using a well-known estimate for the binomial coefficients, we get
  $$ |A_0| < n\exp(nH(N/n)+(n-N)\ln(q-2)). $$
Now, in view of \refe{n-large1} we have
  $$ \frac 1q \le \frac Nn\le \frac 2q \le 1-\frac1q, $$
and therefore, since $H(x)$ is concave and symmetric around the point
$x=1/2$, using \refe{n-large1} once again, from the mean value theorem we
derive
\begin{align*}
  H(N/n) - H(2/q) &= O\big((N/n-2/q)\,H'(1/q)\big) \\
                  &= O\big((\ln q/(nq))^{1/2}\,H'(1/q)\big) \\
                  &= O\big((\ln q)^{3/2}/(nq)^{1/2}\big).
\end{align*}
Hence
\begin{align*}
  nH(N/n)+(n-&N)\ln(q-2) \\
    &= nH(2/q)+n(1-2/q)\ln(q-2) + O\big((n/q)^{1/2}(\ln q)^{3/2}\big) \\
    &= n \Big( \ln q - \frac2q\,\ln 2 \Big)
         + O\big((n/q)^{1/2}(\ln q)^{3/2}\big),
\end{align*}
implying
  $$ |A_0| < \Big( \frac q{2^{2/q}} \Big)^n
                  \exp\big( O\big((n/q)^{1/2}(\ln q)^{3/2}\big) \big). $$
Since $q/2^{2/q}>2$ for $q\ge 3$, we conclude that
  $$ |K_0| \le |A_0|+|B| < \Big( \frac q{2^{2/q}} \Big)^n
                 \exp\big( O\big((n/q)^{1/2}(\ln q)^{3/2}\big) \big)
           = \Big(\frac q{2^{2/q}}\Big)^{n+O\big(\sqrt{n\ln q/q}\big)}. $$

We now use the random rotation trick to replace $K_0$ with a slightly larger
set $K$ containing lines in \emph{all} (not only \emph{popular}) directions.
To this end we chose at random linear automorphisms
 $T_1\longc T_n$ of the vector space $\F_q^n$ and set
  $$ K := T_1(K_0)\cup\dotsb\cup T_n(K_0). $$
Thus, $K$ contains a line in every direction from the set
  $$ D := T_1(D_0)\cup\dotsb\cup T_n(D_0). $$
Choosing a vector $d\in\F_q^n\stm\{0\}$ at random, for each fixed $j\in[1,n]$
the probability that $d\notin T_j(D_0)$ is at most $1/q$, whence the
probability that $d\notin D$ is at most $q^{-n}$. Hence, the probability that
$D\neq\F_q^n\stm\{0\}$ is smaller than $1$, showing that $T_1\longc T_n$ can
be instantiated so that $K$ is a rank-$1$ Kakeya set. It remains to notice
that $|K|\le n|K_0|$.
\end{proof}

\section{Proof of Lemma \refl{universal}.}\label{s:proofs-us}

If $k>n$, then the assertion of the lemma is trivial; suppose, therefore,
that $k\le n$, and let then $m:=\lfl n/k\rfl$. Fix a decomposition
 $\F_q^n=V_0\oplus V_1\oplus\dotsb\oplus V_k$, where $V_0,V_1\longc V_k\le\F_q^n$
are subspaces with $\dim V_i=m$ for $i=1\longc k$, and for each $i\in[0,k]$
let $\pi_i$ denote the projection of $\F_q^n$ onto $V_i$ along the remainder
of the direct sum; thus, $v=\pi_0(v)+\pi_1(v)\longp\pi_k(v)$ for every vector
$v\in\F_q^n$. Finally, let
  $$ U := \{ u\in\F_q^n \colon \pi_i(u)=0
                         \text{ for at least one index } 1\le i\le k \}. $$

A simple computation confirms that the size of $U$ is as claimed. To see why
$U$ contains a translate of every $k$-element subset of $\F_q^n$, given such
a subset $\{a_1\longc a_k\}$ we let $b:=-\pi_1(a_1)-\dotsb-\pi_k(a_k)$ and
observe that, for each $i\in[1,k]$,
  $$ \pi_i(b+a_i) = \pi_i(b) + \pi_i(a_i) = 0, $$
whence $b+a_i\in U$.$\hfill\Box$

\section{Conclusion.}\label{s:conclusion}
For a vector space $V$ and non-negative integer $r\le\dim V$, we defined
\emph{Kakeya sets of rank $r$ in $V$} as those subsets of $V$, containing a
translate of every $r$-dimensional subspace. In the case where $V$ is finite,
we established a lower bound and a number of upper bounds for the smallest
possible size of such sets. Our bounds are close to best possible in the case
where $r$ is bounded and the dimension $\dim V$ does not grow ``too fast''.
They are reasonably tight if $r=1$ and $\dim V$ grows, particularly if $q$ is
odd and not ``too small''. In the case where $\dim V$ grows and $r\ge 2$,
there is no reason to believe our bounds to be sharp; indeed, for $r\gtrsim
q/\log q$ our best upper bound results from a universal set construction
which completely ignores linearity.

Of possible improvements and research directions, the following two seem of
particular interest to us. First, it would be nice to beat the universal set
construction in the regime just mentioned ($\dim V$ grows and $r\ge 2$), or
to show that it produces an essentially best possible bound. Even the case
$q=r=2$ seems non-trivial: we do not know any construction of Kakeya sets of
rank $2$ in $\F_2^n$ of size smaller than $O(2^{3n/4})$, the bound supplied
by $4$-universal sets. The second direction stems from the fact that the
product of Kakeya sets of rank $r$ is a Kakeya set of rank $r$ in the product
space. It is not difficult to derive that, with $\kap_q^{(n)}(r)$ denoting
the smallest possible size of a Kakeya set of rank $r$ in $\F_q^n$, the limit
$\lim_{n\to\infty} \frac1n\,\ln\kap_q^{(n)}(r)$ exists for any fixed $q$ and
$r$. It would be very interesting to find this limit explicitly, even for
just one particular pair $(q,r)\ne(2,1)$. Arguably, most intriguing is the
first non-trivial case $q=3,\ r=1$, due to the fact that lines in $\F_3^r$
are three-term arithmetic progressions.

\section*{Appendix: proof of the lifting lemma.}

We prove here the following lemma, which is a slight extension of
Lemma~\refl{EOTlifting}.
\begin{lemma}\label{t:eot}
Let $n\ge r\ge r_1\ge 1$ be integers and $\F$ a field. Suppose that $K_1$ is
a Kakeya set of rank $r_1$ in $\F^{n-(r-r_1)}$, considered as a subspace of
$\F^n$, and let $K:=K_1\cup(\F^n\stm\F^{n-(r-r_1)})$. Then $K$ is a Kakeya
set of rank $r$ in $\F^n$.
\end{lemma}

\begin{proof}
Suppose that $L\le\F^n$ is a subspace with $\dim L=r$. From
  $$ \dim L + \dim \F^{n-(r-r_1)}
                 = \dim(L+\F^{n-(r-r_1)}) + \dim(L\cap\F^{n-(r-r_1)}) $$
it follows that either $L+\F^{n-(r-r_1)}$ is a proper subspace of $\F^n$, or
$\dim(L\cap\F^{n-(r-r_1)})=r_1$. Observing that if $v\notin
L+\F^{n-(r-r_1)}$, then $v+L$ is disjoint with $\F^{n-(r-r_1)}$, we conclude
that, in either case, there is a translate of $L$, intersecting
$\F^{n-(r-r_1)}$ by a subset of a $r_1$-dimensional subspace. Hence, there is
also a translate of $L$, the intersection of which with $\F^{n-(r-r_1)}$ is
contained in $K_1$. By the construction, this translate of $L$ is contained
in $K$.
\end{proof}


\bigskip

\end{document}